\documentclass[a4 paper,11pt]{article}
\usepackage{stmaryrd}
\usepackage{amsfonts}
\usepackage{bbm}
\usepackage{amscd}
\usepackage{mathrsfs}
\usepackage{latexsym,amssymb,amsmath,amscd,amscd,amsthm,amsxtra,xypic}
\usepackage[dvips]{graphicx}
\usepackage[utf8]{inputenc}
\usepackage[T1]{fontenc}
\usepackage{enumerate}
\usepackage{amssymb}
\usepackage[all]{xy}
\usepackage{ifpdf}
\ifpdf
\usepackage[colorlinks=true,linkcolor=blue,final,backref=page,hyperindex,citecolor=red]{hyperref}
\else
\usepackage[colorlinks,final,backref=page,hyperindex,hypertex]{hyperref}
\fi

\usepackage{nicefrac,mathtools,enumitem}
\usepackage{microtype}
\usepackage{amsfonts}
\usepackage{amssymb}
\usepackage{mathrsfs}
\usepackage{amsmath}
\usepackage{amsthm}
\usepackage{enumerate}
\usepackage{indentfirst}
\usepackage{amsfonts}
\usepackage{amssymb}
\usepackage{mathrsfs}
\usepackage{amsmath}
\usepackage{amsthm}
\usepackage{enumerate}
\usepackage{cite}
\usepackage{mathrsfs}
\usepackage{geometry}
\usepackage{multirow}
\allowdisplaybreaks

\usepackage{dsfont}

\def\<{\langle}
\def\>{\rangle}

\def\c{\cdot}

\newcommand{\bmx}{\begin{pmatrix}}
\newcommand{\emx}{\end{pmatrix}}

\newcommand{\br}[1]{   [ \cdot,    \cdot  ]   }

\newcommand{\rprod}{\triangleleft}  
\newcommand{\lprod}{\triangleright}  
\newcommand{\mprod}{\vartriangle}   

\newcommand {\emptycomment}[1]{}

\newtheorem{thm}{Theorem}[section]
\newtheorem{lem}[thm]{Lemma}
\newtheorem{cor}[thm]{Corollary}
\newtheorem{pro}[thm]{Proposition}
\newtheorem{ex}[thm]{Example}
\newtheorem{exs}[thm]{Examples}

\theoremstyle{definition}
\newtheorem{defi}{Definition}[section]

\theoremstyle{remark}
\newtheorem{rmk}{Remark}[section]

\begin{document}
\title{\sf Anti-Leibniz algebras: A non-commutative version of mock-Lie algebras}
\date{}
\author{\normalsize \bf  Safa Braiek\small{$^{1}$} \footnote {E-mail: safabraiek25@gmail.com},\quad  Taoufik Chtioui{$^2$}\footnote{E-mail: chtioui.taoufik@yahoo.fr}\quad and\quad  Sami Mabrouk\small{$^{3}$} \footnote{  E-mail: mabrouksami00@yahoo.fr, sami.mabrouk@fsgf.u-gafsa.tn (Corresponding author)}}
\date{ {\small{$^{1}$} Faculty of Sciences, University of Sfax,   BP
1171, 3000 Sfax, Tunisia \\{\small$^{2}$}  Mathematics And Applications Laboratory LR17ES11, Faculty of Sciences, Gabes University, Tunisia\\     \small{$^{3}$} Faculty of Sciences, University of Gafsa,   BP
2100, Gafsa, Tunisia
 }}
\maketitle
\begin{center}
    (\textit{Dedicated to  Professor Abdenacer Makhlouf on his 60th birthday})
    \end{center}
\begin{abstract} 
Leibniz algebras are non skew-symmetric  generalization of Lie algebras. 
In this paper we introduce the notion of anti-Leibniz algebras as a "non commutative version" of mock-Lie algebras. 
Low dimensional classification of such algebras is given. 
Then we investigate the notion of averaging operators and more general embedding tensors to build some new algebraic structures, namely anti-associative dialgebras, anti-associative trialgebras and anti-Leibniz trialgebras. 
\end{abstract}
\textbf{Key words and phrases}:  mock-Lie algebra, representation, Leibniz algebra, anti-Leibniz algebra, embedding tensor, averaging operator.\\
\textbf{Mathematics Subject Classification} (2020): 16W10, 	17A60, 17A01, 	17A32, 17A36.

\numberwithin{equation}{section}

\tableofcontents

\section{Introduction and preliminary results}
 Leibniz algebras (also called Loday algebras) are a non-commutative analogue of Lie
algebras. Leibniz algebras were first considered by Bloh  and rediscovered by Loday, motivated by the study of the periodicity in algebraic
 K-theory  \cite{loday}. 
He then defined dialgebra \cite{loday2} as the enveloping algebra of Leibniz algebra by analogy with
 associative algebra as the enveloping algebra of Lie algebra.
A while ago, a new class of algebras emerged in the literature, the so called mock-Lie algebras (or Jacobi-Jordan algebras). A mock-Lie algebra is a 
 commutative algebra satisfying the Jacobi identity. The notion of mock-Lie algebras  was appeared, for the first time, in \cite{B.F}
and since then a lot of works are done on this subject,   (see \cite{P.Z,C.G, Makhlouf1, Makhlouf2}). 
These algebras have two distinct lives: they are peculiar cousins of Lie algebras and they belong to a certain class of Jordan algebras. 

In this paper, we introduce the notion of anti-Leibniz algebras as a non commutative version of mock-Lie algebras. 
Namely, an anti-Leibniz algebra is 
is a pair $(A,\{\cdot,\cdot\})$ consisting of a vector space $A$ endowed with a linear map   $\{\c,\c\}: A \otimes A \to A$  satisfying the anti-Leibniz identity:
    \begin{align}
        \{x,\{y,z\}\}=-\{\{x,y\},z\}-\{y,\{x,z\}\}, \forall x,y,z \in A. 
\end{align}
Similarly as Leibniz algebras are closely related to Lie algebras and associative dialgebras, we extend these relationships to the anti-Leibniz algebra setting. In fact, Lie algebras and associative dialgebras are replaced, respectively, by mock-Lie algebras and anti-associative dialgebras (Definition \ref{antiass dialg}.   
Then, we introduce the notion of anti-associative trialgebras 
(see Definition \ref{antitriass}) and anti-Leibniz trialgebras (see Definition \ref{anti-Leib trialg}).
The main tool to make connections between all these structures is the so called averaging operator and more general embedding tensor (or relative averaging operator). 
 In his celebrated paper \cite{reynolds} written at the end of the 19th century, Reynolds, a pioneer of theoretical fluid dynamics,  introduced an operator that maps
 a function of time and space to its mean over some interval of time. For that
 operator, Reynolds was led to consider the algebraic identity
\begin{align}
   \mathcal K(x\mathcal K (y)+y\mathcal K (x)) = \mathcal K (x)\mathcal K (y)+\mathcal K (\mathcal K (x)\mathcal K (y)).  \label{R}
\end{align}
An operator $\mathcal K$ with property \eqref{R} is called a Reynolds operator. In his study,
 Reynolds also considered averaging operators, i.e., operators $ \mathcal K$ that satisfy the identity
 \begin{align}
      \mathcal K(x  \mathcal K (y)) =  \mathcal K (x) \mathcal K (y).\label{A}
 \end{align}
There now is an extensive literature on averaging operators, motivated to no small degree from their connection to conditional expectation in probability theory. Kampé de Feriet first recognized the importance of studying averaging and Reynolds operators in general, and substantially advanced the topic in \cite{35}. A more algebraic study of these operators was initiated by Dubreil in \cite{21}, while
 the first study of averaging operators by means of functional analysis is due to 
 Birkhoff \cite{11}. Interestingly, the averaging identity \eqref{A} was being studied at about the same time as Kolmogorov’s foundations of probability became known, whereas
 the connection with conditional expectation was made only many years later by
 Moy in \cite{43}.  Recently, averaging operators have been studied for many algebraic structures \cite{1,6,12,pei-rep, Safa,dasmakhlouf}.

The main results of this article are summarized in the following diagram:

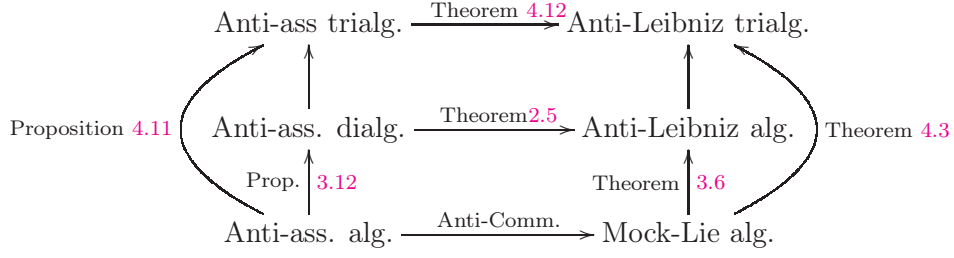
\begin{equation*} 
    \xymatrix{
\text{Anti-ass trialg.}
\ar[rr]^{\text{Theorem \ref{triasstotriLeib}}} & 
&\text{Anti-Leibniz trialg.}  
  \\
\text{Anti-ass. dialg.}\ar[rr]^{\text{Theorem} \text{\ref{antiasstoantiLeib}}}  \ar[u]_{\text{ }}
&   & \text{Anti-Leibniz alg.} \ar[u]^{\text{ }}  \\
\text{Anti-ass. alg.} \ar[u]^{\text{Prop.}}_{\text{\ref{antiasstoantidiass}}}  \ar@/^4pc/[uu]^{\text{Proposition \ref{AntiAsstoAntiAssTri}}} \ar[rr]^{\text{Anti-Comm.}} &    &
 \text{Mock-Lie alg.} \ar[u]^{\text{Theorem}}_{\text{\ref{MLietoALeibniz}}} \ar@/_4pc/[uu]_{\text{Theorem \ref{MLietoAntiLeibTri}}}   
}
\end{equation*}
 
 Unless otherwise specified, all the vector spaces and algebras are finite dimensional over a field $\mathbb{K}$ of characteristic zero. 
 Before introducing the notion of anti-Leibniz algebras and their related structures, we will recall some preliminary concepts. For more details, we refer to  \cite{Agore1,B.F,Attan,Baklouti,LMK}.

Recall, first that a  mock-Lie algebra  is a pair  $(A,\circ)$ consisting of a vector space $A$ together with a commutative product  $\circ: A \otimes A \to A$ satisfying 
\begin{align} 
 &   x\circ(y\circ z)+y\circ(z\circ x)+z\circ(x\circ y)=0,  \quad \mathrm{ (Jacobi\ identity)},  \label{Jacoby id}
\end{align}
   for any  $x,y,z\in A$, or equivalently
\begin{align}
&  x\circ (y \circ z)=-(x\circ y)\circ z -y \circ (x \circ z). \label{antider} 
\end{align}
Note that \eqref{antider} means that the left multiplication $L: A \to End(A)$ defined by $L_x(y)=x \circ y$, is anti-derivation on $A$, namely  
    \begin{equation*}
        L_x (y\circ z)=-L_x(y) \circ z -y\circ L_x(z),\ \forall x,y,z \in A.  
    \end{equation*}
On the other hand, an  anti-associative algebra is a pair $(A, \star)$ consisting of a vector space $A$ together with a product $\star: A \otimes A \to A$ such that the anti-associator vanishes, i.e, 
\begin{align}
  (x,y,z)^a:=(x\star y)\star z+x\star(y\star z)=0,\quad \forall x,y,z
  \in A.   
\end{align}

It is known  that any  anti-associative algebra   $(A, \star)$  gives rise to  a mock-Lie algebra
$(A,\circ)$, via anti-commutator, that is 
\begin{equation}\label{anticom}
    x \circ y := x\star y + y\star x,\; \forall\ x, y \in  A.
\end{equation}

\begin{exs}\label{exmpl} 
 \begin{enumerate}
\item Let $A$ be a $3$-dimensional vector space with basis $\mathcal{B}=\{e_1, e_2, e_3\}$. Then 
$(A,\circ)$ is a mock-Lie algebra where  the product
$\circ$ is defined  by 
$e_1 \circ e_1 = e_2$ and  $e_3 \circ e_3 = e_2.$   
\item Let $A$ be a $4$-dimensional vector space with 
basis $\mathcal{B}=\{e_1, e_2, e_3, e_4\}$. Then 
$(A,\circ)$ is a mock-Lie algebra where  the product
$\circ$ is defined  by 
$e_1 \circ e_1 = e_2$ and  $e_1 \circ e_3 = e_4.$
\end{enumerate}
\end{exs}

 Now, we recall the definition of representations of a mock-Lie algebra.
\begin{defi}
A representation of a mock-Lie  algebra $(A,\circ)$ on a vector space  $V$ is a linear map  $\pi:A \rightarrow End(V)$ satisfying 
\begin{equation}\label{rep}
    \pi(x\circ y)=-\pi(x)\pi(y)-\pi(y)\pi(x),\ \forall x,y \in A.
\end{equation}
\end{defi}

A mock-Lie algebra $(A,\circ )$ with a representation $(V,\pi)$ is called \textsf{MLie-Rep} pair
and refer to it with the pair $(A,V)$. It is obvious that 
 $(A,L)$ is a representation of $A$ on itself called the adjoint representation.

An equivalent  characterisation of representations on mock-Lie algebras is given in the following.

\begin{lem}
Let $(A, \circ)$ be a mock-Lie algebra, $V$ be a vector space and $\pi : A \to End(V)$ a linear map. Then 
$(V, \pi)$ is a representation of $A$ if and only if 
the direct sum $A \oplus V$ together  with the multiplication  defined by
\begin{equation}
    (x+u)\circ_{A\oplus V}(y+v)=x\circ y +\pi(x)v+\pi(y)u,\;\forall x,y\in A,\;\forall u,v\in V, 
\end{equation}
is a mock-Lie algebra. This mock-Lie algebra is called the semi-direct product of $A$ and $V$ and it  is denoted by $A\ltimes_{\pi}V$.
\end{lem}

\emptycomment{
\begin{defi}
Let $(A,\circ)$ be a mock-Lie algebra and two representations $(V_1,\pi_1)$ and $(V_2,\pi_2)$. A linear map $\phi:V_1 \rightarrow V_2$ is  said to be a  morphism of representations if 
\begin{equation}
    \pi_2(x)\circ \phi=\phi \circ \pi_1(x),\;\forall \,x\in A.
\end{equation}
If $\phi$ is bijective, then $(V_1,\pi_1)$ and $(V_2,\pi_2)$ are equivalent $($isomorphic$)$.
\end{defi} }

 In the following we recall the notion of dual representation of mock-Lie algebras. 
Let $(A, \circ)$ be a mock-Lie algebra and $(V, \pi)$ be a representation of $A$. Let $V^*$ be the dual vector space of $V$. Define the linear map $\pi^*:A \rightarrow End(V^*)$ as
\begin{equation}\label{repdual}
    \langle \pi^*(x)u^*,v\rangle =\langle u^*,\pi(x)v\rangle, \quad \forall\, x\in A,\;v\in V,\;u^*\in V^*,
\end{equation}
where $\langle\cdot,\cdot\rangle$ is the usual pairing between $V$ and the dual space $V^*$. 
Then the pair $(V^*,\pi^*)$ is a representation of $A$ on $V^*$.

\begin{defi}
    A representation of an anti-associative algebra $(A,\star)$ is a triple $(V,\varrho,\mu)$, where $V$ is a vector space and $\varrho,\mu:A\rightarrow End(V)$ are linear maps satisfying 
    \begin{align}\label{repantiass}
        &\varrho(x)\varrho(y)=-\varrho(x\star y),\qquad \mu(y)\mu(x)=-\mu(x\star y),\qquad \mu(x)\varrho(y)=-\varrho(y)\mu(x).
    \end{align}
\end{defi}
\begin{ex}
    Let $(A,\star)$ be an  anti-associative algebra and define the linear maps $l,r:A\to End(A)$ by $l(x)(y)=r(y)(x)=x\star y$. Then $(A,l,r)$ is a representation of $A$ on itself, called the adjoint representation.
\end{ex}

\begin{lem}
Let $(A, \star)$ be an anti-associative algebra, $V$ be a vector space and $\varrho,\mu : A \to End(V)$ are linear maps. Then 
$(V, \varrho,\mu)$ is a representation of $A$ if and only if $(A\oplus V, \star_{A\oplus V})$ is an anti-associative algebra, where
\begin{equation}
    (x+u)\star_{A\oplus V}(y+v)=x\star y +\varrho(x)v+\mu(y)u,\;\forall x,y\in A,\;\forall u,v\in V. 
\end{equation}
\end{lem}

\emptycomment{Consider the case when $V=A$ and 
define the  linear map $L^*:A\rightarrow End(A^*)$ by 
\begin{equation}
    \langle L^*({x})(\xi),y\rangle=\langle \xi,L({x})y \rangle,\quad \forall x,y\in A,\;\xi \in A^*.
\end{equation}
Then we have the following:
\begin{cor}
Let $(A,\circ)$ be a mock-Lie algebra and $(A,L)$ be the adjoint representation of $A$. Then $(A^*,L^*)$ is a representation of $(A,\circ)$ on $A^*$ which is called the coadjoint representation.
If there is a mock-Lie algebra structure on the dual space $A^*$, we denote the left multiplication by $\mathcal{L}$.
\end{cor}}

\section{ Anti-Leibniz algebras}\label{Sec3}
This section is devoted to introduce a new algebraic structure called  anti-Leibniz algebra in which the left multiplication is an anti-derivation. Then, we give some examples of  anti-Leibniz algebras. Moreover, we  study  their representation theory and characterize it via semi-direct product.
\subsection{Definition and Examples}In the following,  we introduce the notion of  anti-Leibniz algebras and give some examples.
\begin{defi}\label{antiLeib}
Let $A$ be a vector space and $\{\c,\c\}: A \otimes A \to A$ 
be a linear map. 
The pair $(A,\{\cdot,\cdot\})$ is called a (left) anti-Leibniz algebra if, for any $x,y,z \in A$,
    \begin{align}\label{antiLeibIdentity}
        \{x,\{y,z\}\}=-\{\{x,y\},z\}-\{y,\{x,z\}\},
\quad \textrm{(left anti-Leibniz identity).}
\end{align}
On the other hand, $(A,\{\cdot,\cdot\})$ is called a right anti-Leibniz algebra if 
 \begin{align}
     \{\{x,y\},z\}= -\{\{x,z\},y\}-\{x,\{y,z\}\} ,\ \forall x,y,z \in A. 
\end{align}
\end{defi}

\begin{rmk}
    \begin{enumerate}
        \item In an anti-Leibniz algebra $(A,\{\cdot,\cdot\})$,
        the left multiplication $\mathfrak L: A \to End(A)$ defined by $\mathfrak L(x)y=\{x,y\}$, is anti-derivation. Likewise,  the right multiplication 
        on a right anti-Leibniz algebra $(A,\{\cdot,\cdot\})$,
         defined by $\mathfrak R(x)y=\{y,x\}$,   
        is anti-derivation.
        \item The pair $(A,\{\cdot,\cdot\})$ is called symmetric  anti-Leibniz algebra if it is both a left and right anti-Leibniz algebra. 
    \end{enumerate}
\end{rmk}
\begin{ex}
    Any  mock-Lie algebra is an anti-Leibniz algebra.
\end{ex}
\begin{defi}
    Let $(A,\{\cdot,\cdot\})$ be an anti-Leibniz algebra and $I$ be a sub-vector space of $A$.
    \begin{enumerate}
        \item $I$ is called a subalgebra of $ A $ if $\{I,I\}\subseteq I$.
        \item $I$ is called an ideal of $ A $ if $\{I,A\}\subseteq I$.
    \end{enumerate}
    \end{defi}

\begin{ex}
Let $(A,\{\cdot,\cdot\})$ be an anti-Leibniz. The sub-vector
space $I_A$ spanned by $$\bigg\{\{x,y\}-\{y,x\},\; x,y\in A\bigg\}$$ is an ideal of $A$; in fact $\{I_A,A\}=\{0\}$.
 It is clear that
$A$ is a mock-Lie algebra if and only if $I_A=\{0\}$. Therefore, the quotient algebra $A/I_A$
is a mock-Lie algebra.  The space $I_A$ is called the anti-Leibniz kernel of $A$. 
\end{ex}
\begin{thm}
\begin{enumerate}
    \item 
Any $1$-dimensional anti-Leibniz algebra is trivial.
\item 
    Any $2$-dimensional anti-Leibniz algebra is isomorphic to one of the following 3 algebras (where  $a,b\neq 0$): 
\begin{center}
\begin{tabular}{c|c}
Algebra & Products \\
\hline
\multirow{3}{*}{$A_1$} & $\{e_1,e_1\}=-ae_1-\frac{a^2}{b}e_2$\\&             
 $\qquad\{e_1,e_2\}=\{e_2,e_1\}=be_1+ae_2$\\& $\{e_2,e_2\}=-\frac{b^2}{a}e_1-be_2$\\

\hline
$A_2$ & $\{e_2,e_2\}=e_1$\\
\hline
$A_3$ & $\{e_1,e_1\}=e_2$ \\
\hline
\end{tabular}
\end{center} 
    \end{enumerate}
\end{thm}

The following  observation is obvious. 
\begin{pro}
    If $(A,\{\cdot,\cdot\})$ is a  anti-Leibniz algebra, then $(A,\{\cdot,\cdot\}^{op})$ is a right  anti-Leibniz algebra, where $\{x,y\}^{op}=\{y,x\}$, for all $x,y\in A$.
\end{pro}
In the following,  we introduce the notion of anti-associative dialgebras which is a duplication of anti-associative algebras.
Let $(A,\triangleright,\triangleleft)$ be a  dialgebra. Define the left, right and inner antiassociators as follow
\begin{align*}
&(x,y,z)^a_\triangleleft=(x\triangleleft y)\triangleleft z+x\triangleleft(y \triangleleft z),\\
&(x,y,z)^a_\triangleright =(x\triangleright y)\triangleright z+x\triangleright(y \triangleright z),\\
&(x,y,z)^a_\times =(x\triangleright y)\triangleleft z+x\triangleright(y \triangleleft z).
\end{align*}
In addition, define,  the anti-dicommutator  as
\begin{align}
 \{x,y\}= x \triangleright y+ y \triangleleft x ,\quad \forall x,y\in A.  \label{antidicomm}
\end{align}
\begin{defi} \label{antiass dialg}
    An anti-associative dialgebra is a dialgebra  $(A,\triangleright,\triangleleft)$
  satisfying
\begin{align}
\label{antiass dialg1}(x,y,z)^a_\triangleleft=& 0,\\
\label{antiass dialg2}(x,y,z)^a_\triangleright=&0, \\
(x,y,z)^a_\times=&0 \label{antiass dialg3}
\end{align}
for all $x,y,z \in A$. 
\end{defi}
Any anti-associative dialgebra gives rise to an anti-Leibniz algebra via anti-dicommutator.
\begin{thm}\label{antiasstoantiLeib}
 Let $(A,\lprod,\rprod)$ be an anti-associative dialgebra. Then, $(A,\{\cdot,\cdot\})$ is a  anti-Leibniz algebra, where the bracket $\{\cdot,\cdot\}$ is defined in \eqref{antidicomm}.
\end{thm}
\begin{proof}
    Let $(A,\lprod,\rprod)$ be an anti-associative dialgebra. then for any $x,y,z\in A$, we have
    \begin{align*}
        \{x,\{y,z\}\}+\{\{x,y\},z\}+\{y,\{x,z\}\}=&\{x,y\lprod z+z\rprod y\}+\{x\lprod y+y\rprod x,z\}+\{y,x\lprod z+z\rprod x\}\\=&x\lprod (y\lprod z)+x\lprod (z\rprod y)+(y\lprod z)\rprod x+(z\rprod y)\rprod x\\&+(x\lprod y)\lprod z+(y\rprod x)\lprod z+z\rprod(x\lprod y)+z\rprod (y\rprod z)\\&+y\lprod (x\lprod z)+y\lprod(z\rprod x)+(x\lprod z)\rprod y+(z\rprod x)\rprod y\\=&(x,y,z)^a_\triangleleft+(x,y,z)^a_\times+(y,z,x)^a_\times+(z,y,x)^a_\triangleright\\&+(y,x,z)^a_\triangleleft+(z,x,y)^a_\triangleright\\
        =&0.
    \end{align*}
    Hence, $(A,\{\cdot,\cdot\})$ is an anti-Leibniz algebra.
\end{proof}

\subsection{Representations of  anti-Leibniz algebras}

\begin{defi}\label{repantiLeib}
A  representation of an anti-Leibniz algebra $(A,\{\cdot,\cdot\})$ is a triple $(V,\mathfrak l,\mathfrak r)$, where $V$ is a vector space, $\mathfrak l,\mathfrak r:A\to End(V)$ are linear maps satisfying: 
\begin{align}
\label{rep-1}\mathfrak l(\{x,y\})=&-\mathfrak l(x)\mathfrak l(y)-\mathfrak l(y)\mathfrak l(x),\\
\label{rep-2}\mathfrak r(\{x,y\})=&-\mathfrak l(x)\mathfrak r(y)-\mathfrak r(y)\mathfrak l(x),\\
\label{rep-3}\mathfrak r(y)  \mathfrak l(x)=&\mathfrak r(y) \mathfrak r(x)
\end{align}
for all $x,y\in A$.
\end{defi}
\begin{rmk}
    According to Eqs \eqref{rep-2} and \eqref{rep-3} we have the following identity
    \begin{align}
\label{rep-22}\mathfrak r(\{x,y\})=&-\mathfrak l(x)\mathfrak r(y)-\mathfrak r(y) \mathfrak r(x)
\end{align}
    for all $x,y\in A.$
    
\end{rmk}

\begin{ex}
 Let $(A,\{\cdot,\cdot\})$ be an anti-Leibniz algebra. Define the left and right  multiplication operators  $L, R:A\longrightarrow End(A)$ by $L_x(y)=\{x,y\}$ and $R_x(y)=\{y,x\}$, for all $x,y\in A$.  Then $(A;L,R)$ is a representation of $(A,\{\cdot,\cdot\})$, which is called the adjoint representation.   
\end{ex}

\begin{pro}
Let $(A,\{\cdot,\cdot\})$ be an anti-Leibniz algebra and let $V$ be a vector space. Suppose that $\mathfrak l, \mathfrak r: A \to End(V)$ are two linear maps. Then 
$(V, \mathfrak l, \mathfrak r)$ is a representation of $A$ if and only if 
the direct sum $A \oplus V$ together  with the multiplication  defined by
\begin{equation}
    \{x+u,y+v\}_{A\oplus V}=\{x, y\} +\mathfrak l(x)v+\mathfrak r(y)u, 
\end{equation}
for all $x,y\in A$ and  $u,v\in V$, is an anti-Leibniz algebra. This  anti-Leibniz algebra is called the semi-direct product of $A$ and $V$ and it  is denoted by $A\ltimes_{\mathfrak l, \mathfrak r}V$.    
\end{pro}
\begin{proof}
Let $(A,\{\cdot,\cdot\})$ be an anti-Leibniz algebra. Then, using Definitions \ref{antiLeib} and \ref{repantiLeib}, we have for any $x,y,z\in A$ and $u,v,w\in V$
\begin{align*}
   & \{x+u,\{y+v,z+w\}_{A\oplus V}\}_{A\oplus V}+\{\{x+u,y+v\}_{A\oplus V},z+w\}_{A\oplus V}+\{y+v,\{x+u,z+w\}_{A\oplus V}\}_{A\oplus V}\\=&\{x+u,\{y,z\}+\mathfrak l(y)w+\mathfrak r(z)v\}_{A\oplus V}+\{\{x,y\}+\mathfrak l(x)v+\mathfrak r(y)u,z+w\}_{A\oplus V}\\&+\{y+v,\{x,z\}+\mathfrak l(x)w+\mathfrak r(z)u\}_{A\oplus V}\\=&\{x,\{y,z\}\}+\mathfrak l(x)\mathfrak l(y)w+\mathfrak l(x)\mathfrak r(z)v+\mathfrak r(\{y,z\})u+\{\{x,y\},z\}+\mathfrak l(\{x,y\})w+\mathfrak r(z)\mathfrak l(x)v+\mathfrak r(z)\mathfrak r(y)u\\&+\{y,\{x,z\}\}+\mathfrak l(y)\mathfrak l(x)w+\mathfrak l(y)\mathfrak r(z)u+\mathfrak r(\{x,z\})v\\=&0.
\end{align*}
Thus, $(A\oplus V,\{\cdot,\cdot\}_{A\oplus V})$ is an anti-Leibniz algebra.
\end{proof}


Let $(V,\mathfrak l,\mathfrak r)$ be a representation of an anti-Leibniz algebra $(A,\{\cdot,\cdot\})$. Define the  two linear maps $\mathfrak l^*,\mathfrak r^*:A\longrightarrow End(V^*)$, respectively by
\begin{eqnarray}
\langle \mathfrak l^*_{x}\xi,v\rangle=\langle \xi,\mathfrak l(x)v\rangle,\quad\langle \mathfrak r^*_{x}\xi,v\rangle=\langle \xi,\mathfrak r(x)v\rangle,\quad\forall x\in A,v\in V,\xi\in V^*.
\end{eqnarray}

\begin{lem}\label{lem:dualrep}Under the above notations,
 $\big(V^*,\mathfrak l^*,\mathfrak l^*-\mathfrak r^*\big)$ is a representation of   $(A,\{\cdot,\cdot\})$, which is called the  dual representation of  $(V,\mathfrak l,\mathfrak r)$.
\end{lem}
\begin{proof}
By \eqref{rep-1}, for all $x,y\in A,~v\in V$ and $\xi\in V^*$,  we have
\begin{eqnarray*}
\langle\mathfrak l^*(\{x,y\})\xi,v\rangle&=&\langle \xi,\mathfrak l(\{x,y\})v\rangle\\&=&-\langle \xi,\mathfrak l(x)\mathfrak l(y)v+\mathfrak l(y)\mathfrak l(x)v\rangle\\
&=&-\langle \mathfrak l^*(y)\mathfrak l^*(x)\xi,v\rangle-\langle \mathfrak l^*(x)\mathfrak l^*(y)\xi,v\rangle\\
&=&\langle- \mathfrak l^*(y)\mathfrak l^*(x)\xi- \mathfrak l^*(x)\mathfrak l^*(y)\xi,v\rangle.
\end{eqnarray*}
Thus, we have $\mathfrak l^*(\{x,y\})=- \mathfrak l^*(y)\mathfrak l^*(x)-\mathfrak l^*(x)\mathfrak l^*(y)$. By \eqref{rep-1} and \eqref{rep-2}, we have
\begin{eqnarray*}
&&\langle \big(\mathfrak l^*(\{x,y\})-\mathfrak r^*(\{x,y\})\big)\xi,v\rangle\\&=&\langle \xi,\mathfrak l(\{x,y\})v-\mathfrak r(\{x,y\})v\rangle\\
&=&\langle \xi,-\mathfrak l(x)\mathfrak l(y)v-\mathfrak l(y)\mathfrak l(x)v+\mathfrak l(x)\mathfrak r(y)v+\mathfrak r(y)\mathfrak l(x)v\rangle\\
   &=&-\langle \mathfrak l^*(y)\mathfrak l^*(x)\xi,v\rangle-\mathfrak l^*(x)\mathfrak l^*(y)\xi,v\rangle\\
      &&+\langle \mathfrak r^*(y)\mathfrak l^*(x)\xi,v\rangle+\langle \mathfrak l^*(x)\mathfrak r^*(y)\xi,v\rangle\\
        &=&-\langle \mathfrak l^*(x)(\mathfrak l^*(y)-\mathfrak r^*(y))\xi,v\rangle-\langle (\mathfrak l^*(y)-\mathfrak r^*(y))\mathfrak l^*(x)\xi,v\rangle\\
        &=&\langle -\mathfrak l^*(x)(\mathfrak l^*(y)-\mathfrak r^*(y))\xi- (\mathfrak l^*(y)-\mathfrak r^*(y))\mathfrak l^*(x)\xi,v\rangle.
\end{eqnarray*}
Thus, we have $\mathfrak l^*(\{x,y\})-\mathfrak r^*(\{x,y\})=-\mathfrak l^*(x)(\mathfrak l^*(y)-\mathfrak r^*(y))- (\mathfrak l^*(y)-\mathfrak r^*(y))\mathfrak l^*(x)$.
In addition, using by \eqref{rep-3}, we have 
\begin{eqnarray*}
&&\langle (\mathfrak l^*(y)-\mathfrak r^*(y))(\mathfrak l^*(x)\xi),v\rangle\\&=&\langle \xi,(\mathfrak l(x) \mathfrak l(y)-\mathfrak l(x) \mathfrak r(y))v\rangle\\
                &=&\langle \xi,(\mathfrak l(x) \mathfrak l(y)-\mathfrak r(x)\mathfrak l(y)-\mathfrak l(x) \mathfrak r(y)+\mathfrak r(x) \mathfrak r(y))v\rangle\\
                  &=&\langle \xi,(\mathfrak l(x)-\mathfrak r(x))(\mathfrak l(y)-\mathfrak r(y))v\rangle\\
           &=&\langle (\mathfrak l^*(y)-\mathfrak r^*(y))(\mathfrak l^*(x)-\mathfrak r^*(x))\xi,v\rangle
\end{eqnarray*}
for all $x,y\in A,v\in V$ and $\xi\in V^*$. 
Thus, we have
$$\big(\mathfrak l^*(y)-\mathfrak r^*(y)\big) \mathfrak l^*(x)=\big(\mathfrak l^*(y)-\mathfrak r^*(y)\big) \big(\mathfrak l^*(x)-\mathfrak r^*(x)\big).$$
Therefore  $\big(V^*,\mathfrak l^*,\mathfrak l^*-\mathfrak r^*\big)$ is a representation of   $(A,\{\cdot,\cdot\})$.
\end{proof}

 Define two linear maps $L^*,R^*:A\longrightarrow End(A^*)$ with $x\longrightarrow L^*_x$ and $x\longrightarrow R^*_x$ respectively by
\begin{eqnarray}
\langle L^*_{x}(\xi),y\rangle=\langle \xi,\{x,y\}\rangle,\quad\langle R^*_{x}(\xi),y\rangle=\langle \xi,\{y,x\}\rangle,\quad\forall x,y\in A,\quad \xi\in A^*.
\end{eqnarray}
\begin{cor}
  Under the above notations,
 $\big(A^*, L^*,L^*-R^*\big)$ is a representation of   $(A,\{\cdot,\cdot\})$, which is called the  coadjoint representation of  $(A;L,R)$.
\end{cor}

\section{Embedding tensor on mock-Lie algebras and  anti-Leibniz algebras}\label{Sec4}

The notion of  embedding tensors on  Lie algebras was introduced in \cite{Kotov-Strobl}.  In this section, we extend this notion to the mock-Lie algebra setting. Then we give two equivalent characterization theorems via graphs and Nijenhuis operators.
\begin{defi} 
Let $(A,\circ)$ be a  mock-Lie algebra and $(V,\pi)$ be a representation. 
   A linear map $\mathcal K: V \to A$ is called an embedding tensor  if it satisfies 
      \begin{align}\label{EmTen}
  & \mathcal K(u)\circ \mathcal K(v)=\mathcal K(\pi(\mathcal K(u))v),\ \forall u,v \in V.
      \end{align}
\end{defi}
The embedding tensor of a mock Lie algebra $A$ with respect to the adjoint representation $(A,L)$ is called an averaging operator. In this case the identity \eqref{EmTen} can be written as
\begin{align*}
  & \mathcal K(x)\circ \mathcal K(y)=\mathcal K(\mathcal K(x)\circ y)=\mathcal K(x\circ\mathcal K(y)),\, \forall\, x,y \in A.
      \end{align*}
      
\begin{ex}
  Let $(A,\circ)$ be a mock-Lie algebra. A derivation for a mock-Lie algebra is a linear map $d:A\to A$ if $d(x\circ y)=x\circ d(y)+d(x)\circ y$. If $d^2=0$, then $d$ is an embedding tensor on the mock-Lie algebra $(A,\circ)$ with respect to the adjoint representation.  
\end{ex} 
 \begin{ex}
     Let $(A,V)$ be a {\sf MLie-Rep} pair and let $f:V\rightarrow A$ be an $A$-module map, meaning that $f(\pi(x)u)=x\circ f(u)$. Then, we have
     \begin{align*}
&f(\pi(f(u))v)=f(u)\circ f(v),\quad\forall   u,v\in V.
     \end{align*}
     Thus, $f$ is an embedding tensor on the mock-Lie algebra $(A,\circ)$. 
 \end{ex}
Let $(A,V)$ be a \textsf{MLie-Rep} pair.  On the direct sum vector space $A\oplus V$,  define the product   $\circ_{A\oplus V}$    by
\begin{eqnarray}
   \{ x+u,y+v\}_{A\oplus V}&=&x\circ y+\pi(x)v,\quad\forall x,y\in A,\; u,v\in V.
\end{eqnarray}

\begin{pro}\label{hsdirect Lie}
  With the above notations, $(A\oplus V,\{\cdot,\cdot\}_{A\oplus V})$ is  anti-Leibniz algebra, called hemisemi-direct product  anti-Leibniz algebra and denoted by $A\oplus_{ALeib} V$.
\end{pro}
\begin{proof}
Using \eqref{rep} and \eqref{antiLeibIdentity}, for all $x,y,z\in A$ and $u,v,w\in V$, we have
\begin{align*}
  &\{x+u,\{y+v,z+w\}_{A\oplus V}\}_{A\oplus V}+\{\{x+u,y+v\}_{A\oplus V},z+w\}_{A\oplus V}\\&+\{y+v,\{x+u,z+w\}_{A\oplus V}\}_{A\oplus V}\\=& x\circ(y\circ z)+\pi(x)\pi(y)w+(x\circ y)\circ z+\pi(x\circ y)w+y\circ (x\circ z)+\pi (y)\pi(x)w\\=&0.
\end{align*}
Hence,  $(A\oplus V,\{\cdot,\cdot\}_{A\oplus V})$ is an anti-Leibniz algebra.
\end{proof}
 \begin{thm}\label{graphMLie}
Let $(A,V)$ be a \textsf{MLie-Rep} pair. A linear map $\mathcal{K}:V\to A$ is an embedding tensor on the mock-Lie algebra $(A,\circ)$ if and only if  the graph ${\sf Gr}(\mathcal{K}) = \{\mathcal{K}u+u\ |\ u \in V\}$ is a anti-Leibniz subalgebra of the hemisemi-direct product $A\oplus_{ALeib}V$.
\end{thm}
\begin{proof}
Let $\mathcal{K}:V\to A$ be a linear map and $(\mathcal{K}u+u), \, (\mathcal{K}v+v)\in {\sf Gr}(\mathcal{K})$, then we have 
\begin{align*}
    \{\mathcal{K}u+u,\mathcal{K}v+v\}_{A\oplus V}=\mathcal{K}u\circ \mathcal{K}v+\pi(\mathcal{K}u)v.
\end{align*}
Hence, the graph ${\sf Gr}(\mathcal{K})$ is a subalgebra of the hemisemi-direct product anti-Leibniz algebra $A\oplus_{ALeib}V$ if and only if $\mathcal{K}$ satisfies $\mathcal{K}u\circ\mathcal{K}v=\mathcal{K}(\pi(\mathcal{K}u)v)$, which implies that $\mathcal{K}$ is an embedding tensor of the mock-Lie algebra $(A,\circ)$.
\end{proof}
Recall that a Nejinhuis operator on an  algebra $(A,\bullet)$ is a linear map $N:A\to A$, such that for all $x,y\in A,$
\begin{equation}
    \label{NijCond}
    N(x)\bullet N(y)=N\big(N(x)\bullet y+
    x
    \bullet N(y)-N(x\bullet y)\big).
\end{equation}
\begin{pro}
   Let $(A,V)$ be a \textsf{MLie-Rep} pair. A linear map $\mathcal K: V \to A$ is an embedding tensor if and only if the map
$$N_{\mathcal K}: (A \oplus V) \to (A \oplus V),\ x+u \mapsto \mathcal K(u)$$
is a Nijenhuis operator on the  hemisemi-direct product $A \oplus_{ALeib} V$.
\end{pro}
\begin{proof}
    Let $\mathcal{K}:V\to A$ be a linear map. For all $x,y\in A$ and $u,v\in V$,  we have
    \begin{align*}
        &\{N_{\mathcal{K}}(x+u),N_{\mathcal{K}}(y+v)\}_{A\oplus V}-N_{\mathcal{K}}\big(\{N_{\mathcal{K}}(x+u),y+v\}_{A\oplus V}+\{x+u,N_{\mathcal{K}}(y+v)\}_{A\oplus V}\\&-N_{\mathcal{K}}(\{x+u,y+v\}_{A\oplus V})\big)\\=&\mathcal{K}(u)\circ \mathcal{K}(v)-N_{\mathcal{K}}\big(\mathcal{K}(u)\circ y+\pi(\mathcal{K}(u))v+x\circ\mathcal{K}(v)-\mathcal{K}(\pi(x)v)\big)\\=&\mathcal{K}(u)\circ \mathcal{K}(v)-\mathcal{K}(\pi(\mathcal{K}(u))v).
    \end{align*}
    Therefore $N_{\mathcal{K}}$ is a Nijenhuis operator on the  hemisemi-direct product $A \oplus_{ALeib} V$ if and only if $\mathcal{K}$ is an embedding tensor on the mock-Lie algebra $(A,\circ)$. 
\end{proof}
 
\begin{thm}\label{MLietoALeibniz}
Let $\mathcal{K}:V\rightarrow A$ be an embedding tensor  on a \textsf{MLie-Rep} pair $ (A,V)$,   then $(V,\{\cdot,\cdot\}_\mathcal{K} )$ is a    anti-Leibniz algebra, where
    \begin{eqnarray*}
\{u,v\}_\mathcal{K} &=&\pi(\mathcal{K}(u))(v),\quad \forall u,v\in V.
    \end{eqnarray*}
\end{thm}

\begin{proof}
    Let $u,v,w\in V$.  Then
    $\mathcal K\{u,v\}_{\mathcal{K}}=\mathcal K(u)\circ\mathcal K(v)$. On the other hand, we have    \begin{eqnarray*}
       & & \{u,\{v,w\}_\mathcal{K} \}_\mathcal{K} +\{\{u,v\}_\mathcal{K} ,w\}_\mathcal{K} +\{v,\{u,w\}_\mathcal{K} \}_\mathcal{K} \\
       & =& \pi(\mathcal K\{u,v\}_{\mathcal{K}})w+\pi(\mathcal K(u))\pi(\mathcal K(v))w+\pi(\mathcal K(v))\pi(\mathcal K(u))w\\
        &=& \pi(\mathcal K(u)\circ\mathcal K(v))w+\pi(\mathcal K(u))\pi(\mathcal K(v))w+\pi(\mathcal K(v))\pi(\mathcal K(u))w\\
        &\overset{\eqref{rep}}{=} &0.
    \end{eqnarray*}
    Then $(V,\{\cdot,\cdot\}_\mathcal{K} )$ is an anti-Leibniz algebra.
\end{proof}
\begin{cor}
    Let $(A,\circ)$ be a mock-Lie algebra and $\mathcal{K}:A\to A$ be an averaging operator. Then, $A$ is also equipped with an anti-Leibniz structure defined by $$\{x,y\}_\mathcal{K}=\mathcal{K}(x)\circ y.$$
\end{cor}
\begin{ex} 
According to Example \ref{exmpl}
(2), $(A,\circ)$ is a $4$-dimensional mock-Lie algebra. We define an averaging operator $\mathcal{K}$ on $(A,\circ)$ which acts on the basis $\mathcal{B}=\{e_1, e_2, e_3, e_4\}$  as follow: 
$$\mathcal{K}(e_1)=\mathcal{K}(e_3)=e_1, \qquad \mathcal{K}(e_2)=\mathcal{K}(e_4)=e_2.$$
Then, by applying the previous corollary, $(A,\{\cdot,\cdot\})$ becomes an anti-Leibniz algebra,  with the bracket defined by     
\begin{align*}
    &\{e_1,e_1\}= e_2,\qquad \{e_1,e_3\}= e_4,\\
    &\{e_3,e_1\}= e_2,\qquad \{e_3,e_3\}= e_4.
\end{align*}
\end{ex}

Let $\mathcal K:V\longrightarrow A$ be an  embedding tensor on a \textsf{MLie-Rep} pair $ (A,V)$.
According to Theorem \ref{MLietoALeibniz}, the bracket  $\{u,v\}_\mathcal{K}:=\pi(\mathcal Ku)v$ carries an anti-Leibniz algebra structure on $V$. Define $\pi^L:V\longrightarrow End(A)$ and $\pi^R:V\longrightarrow End(A)$ by
\begin{equation}
\pi^L(u)y=\mathcal Ku\circ y,\quad \pi^R(v)x=x\circ \mathcal Kv+\mathcal K(\pi(x)v)\quad\forall x,y\in V.
\end{equation}
Then we have the following observation. 
\begin{pro}
The triple  $(A;\pi^L,\pi^R)$ is a representation of the anti-Leibniz algebra $(V,\{\cdot,\cdot\}_\mathcal{K})$.
\end{pro}

\begin{proof}
Straightforward. 
\end{proof}

   Let $(A,\star)$ be an anti-associative algebra and $(V,\varrho,\mu)$ be a representation. 
   A linear map $\mathcal K: V \to A$ is called an embedding tensor  if it satisfies 
      \begin{align}\label{EmTenantiass}
  & \mathcal K(u)\star \mathcal K(v)=\mathcal K(\varrho(\mathcal K(u))v)=\mathcal K(\mu(\mathcal K(v))u),\ \forall u,v \in V.
      \end{align} 
      The embedding tensor of an anti-associative algebra $A$ with respect to the adjoint representation $(A,l,r)$ is called an averaging operator. In this case the identity \eqref{EmTenantiass} can be written as
\begin{align*}
  & \mathcal K(x)\star \mathcal K(y)=\mathcal K(\mathcal K(x)\star y)=\mathcal K(x\star\mathcal K(y)),\, \forall\, x,y \in A.
      \end{align*}
\begin{pro}\label{embtensAAandMC}
    Let $(V,\varrho,\mu)$ be a representation of an anti-associative algebra $(A,\star)$. Then $(V,\pi=\varrho+\mu)$ is a representation of the induced mock-Lie algebra $(A,\circ)$ given in \eqref{anticom}.
Moreover, if $\mathcal{K}:V\to A$ is an embedding tensor on the anti-associative algebra $(A,\star)$ with respect to the representation $(V,\varrho,\mu)$. Then $\mathcal{K}$ is an embedding tensor on the mock-Lie algebra $(A,\circ)$ with respect to the representation $(V,\pi)$.
\end{pro}

\begin{proof}
    For any $x,y\in A$ and $u\in V$, by \eqref{repantiass} we have
    \begin{align*}
        &\pi(x\circ y)u+\pi(x)(\pi(y)u)+\pi(y)(\pi(x)u)\\=&(\varrho+\mu)(x\star y+y\star x)u+(\varrho+\mu)(x)((\varrho+\mu)(y)u)+(\varrho+\mu)(y)((\varrho+\mu)(x)u)\\=&\varrho(x\star y)u+\varrho(y\star x)u+\mu(x\star y)u+\mu(y\star x)u+\varrho(x)\varrho(y)u+\varrho(x)\mu(y)u+\mu(x)\varrho(y)u+\mu(x)\mu(y)u\\&+\varrho(y)\varrho(x)u+\varrho(y)\mu(x)u+\mu(y)\varrho(x)u+\mu(y)\mu(x)u\\=&0.
    \end{align*}
    Thus $(V,\pi=\varrho+\mu)$ is a representation of $(A,\circ)$.
On the other hand, if $\mathcal{K}$ is an embedding tensor on the anti-associative algebra $(A,\star)$, then for any $u,v\in V$, we have
    \begin{align*}
        \mathcal{K}(u)\circ\mathcal{K}(v)&=\mathcal{K}(u)\star\mathcal{K}(v)+\mathcal{K}(v)\star\mathcal{K}(u)\\&=\mathcal{K}(\varrho(\mathcal{K}(u))v)+\mathcal{K}(\mu(\mathcal{K}(u))v)\\&=\mathcal{K}(\pi(\mathcal{K}(u))v).
    \end{align*}
    Hence, $\mathcal{K}$ is also an embedding tensor on the induced mock-Lie algebra $(A,\circ)$ with respect to $(V,\pi)$.
\end{proof}
\begin{cor}
    Any averaging operator $\mathcal{K}:A\to A$ on an anti-associative algebra $(A,\star)$ is an averaging operator on the induced mock-Lie algebra $(A,\circ)$.
\end{cor}
\begin{pro}\label{antiasstoantidiass}
    Let $\mathcal{K}:V\rightarrow A$ be an embedding tensor on an anti-associative algebra $(A,\star)$ with respect to the representation $(V,\varrho,\mu)$. Then, $V$ carries an anti-associative dialgebra structure given by 
    \begin{align*}
        u\triangleright_{\mathcal{K}} v=\varrho(\mathcal{K}u)v,\quad u\triangleleft_{\mathcal{K}} v=\mu(\mathcal{K}v)u,\, \forall \ u,v\in V. 
    \end{align*}
\end{pro}
\begin{proof}
    Using Eqs \eqref{repantiass} and \eqref{EmTenantiass}, for any $u,v,w\in V$, we have 
    \begin{align*}
        (u\triangleright_\mathcal{K}v)\triangleright_\mathcal{K} w+u \triangleright_\mathcal{K}(v \triangleright_\mathcal{K} w)&=\varrho(\mathcal{K}(\varrho(\mathcal{K}u)v))w+\varrho(\mathcal{K}u)(\varrho(\mathcal{K}v)w)\\&=\varrho(\mathcal{K}(u)\star \mathcal{K}(v))w-\varrho(\mathcal{K}(u)\star \mathcal{K}(v))w=0,
    \end{align*}
    on the other hand we have
    \begin{align*}
     (u\triangleleft_\mathcal{K}v)\triangleleft_\mathcal{K} w+u \triangleleft_\mathcal{K}(v \triangleleft_\mathcal{K} w) &=-\mu(\mathcal{K}(v)\star \mathcal{K}(w))u+\mu(\mathcal{K}(v)\star \mathcal{K}(w))u=0.  
    \end{align*}
    Similarly, one can check the Eq \eqref{antiass dialg3}. Thus, $(V,\triangleright_\mathcal{K},\triangleleft_\mathcal{K})$ is an anti-associative dialgebra.
\end{proof}
\section{Anti-Leibniz trialgebras}\label{Sec5}
In this section,  we make a triplication of an anti-Leibniz algebra obtaining what we will  call anti-Leibniz trialgebra. 
To do this, we need a homomorphic embedding tensor on
a mock-Lie algebra which allows the replicating of the multiplication into three parts with some compatibilities.

\begin{defi}  
   Let $(A,\circ_A)$ and $(B,\circ_B)$ be two mock-Lie algebras. We  say that $A$ acts on $B$ if there is a linear map $\pi:A\to End(B)$ such that $(B,\pi)$ is a representation of $A$ and for any $x\in A,\  a,b\in B$, we have
    \begin{equation}\label{actLie}
        \pi(x)(a\circ_Bb)+
        a\circ_B\pi(x)b+b\circ_B\pi(x)a=0.
    \end{equation}
    \end{defi}
    
    A \textsf{MLie-Act} pair is a mock-Lie algebra $(A,\circ_A)$ with an action $(B,\pi,\circ _B)$
and we refer to such a  pair by $(A,B)$.
\begin{pro}
   The  tuple $(B,\pi,\circ_B)$ is an action  of  a mock-Lie  algebra $(A,\circ_A)$
if and only if  $A \oplus B$  carries a mock-Lie structure with multiplication given by
\begin{align*}
&(x + a)\circ_{A \oplus B}  (y + b )=  x \circ_A  y  + \pi(x)b +\pi(y)a +a\circ_B   b 
\end{align*}
 for all  $x,y \in A$ and $ a,b \in B,$ which  is called the semi-direct product of $A$ with $B$.
\end{pro}

\begin{proof}
Let $(A,\circ_A)$ and $(B,\circ_B)$ be two mock-Lie algebras. Then for all $x,y,z\in A$ and $a,b,c\in B$, we have
\begin{align*}
    &(x+a)\circ_{A\oplus B} ((y+b)\circ_{A\oplus B}(z+c))+(y+b)\circ_{A\oplus B}((z+c))\circ_{A\oplus B}(x+a))\\&+(z+c)\circ_{A\oplus B}((x+a)\circ_{A\oplus B}(y+b))\\=&(x+a)\circ_{A\oplus B}(y\circ_A z+\pi(y)c+\pi(z)b+b\circ_B c)\\&+(y+b)\circ_{A\oplus B}(z\circ_A x+\pi(z)a+\pi(x)c+c\circ_B a)\\&+(z+c)\circ_{A\oplus B}(x\circ_A y+\pi(x)b+\pi(y)a+a\circ_B b)\\=&x\circ_A (y\circ_A z)+\pi(x)(\pi(y)c+\pi(z)b+b\circ_B c)+\pi(y\circ_A z)a+a\circ_B (\pi(y)c+\pi(z)b+b\circ_B c)\\&+y\circ_A (z\circ_A x)+\pi(y)(\pi(z)a+\pi(x)c+c\circ_B a)+\pi(z\circ_A x)b+b\circ_B(\pi(z)a+\pi(x)c+c\circ_B a)\\&+z\circ_A (x\circ_A y)+\pi(z)(\pi(x)b+\pi(y)a+a\circ_B b)+\pi(x\circ_A y)c +c\circ_B (\pi(x)b+\pi(y)a+a\circ_B b)\\=&\pi(x)(\pi(y)c+\pi(z)b+b\circ_B c)+\pi(y\circ_A z)a+a\circ_B (\pi(y)c+\pi(z)b)\\&+\pi(y)(\pi(z)a+\pi(x)c+c\circ_B a)+\pi(z\circ_A x)b+b\circ_B(\pi(z)a+\pi(x)c)\\&+\pi(z)(\pi(x)b+\pi(y)a+a\circ_B b)+\pi(x\circ_A y)c +c\circ_B (\pi(x)b+\pi(y)a).
\end{align*}
Then, $(A\oplus B,\circ_{A\oplus B})$ is a mock-Lie algebra if and only if $(B,\pi)$ is a representation of $(A,\circ_A)$ and $$\pi(x)(a\circ_Bb)+
        a\circ_B\pi(x)b+b\circ_B\pi(x)a=0,\,\forall\, x\in A,\, a,b\in B.$$
   That is $(A\oplus B,\circ_{A\oplus B})$ is a mock-Lie algebra if and only if  $(B,\pi,\circ_B)$ is an action of $(A,\circ_A)$. 
\end{proof}

In the following, we introduce the notion of homomorphic embedding tensor on mock-Lie algebra with a given action.
\begin{defi}
       A linear map $\mathcal H: B \to A$ is called a  homomorphic embedding tensor on the \textsf{MLie-Act} pair $(A,B)$ if it's an embedding tensor and a morphism.    
\end{defi}
A homomorphic embedding tensor on a mock Lie algebra $A$ with respect to the adjoint representation $(A,L)$ is called a homomorphic averaging operator.
\begin{ex}
    A crossed module of mock-Lie algebra is a quadruple $(A, B, \partial,\pi)$ in which $(A,\circ_A)$ is a mock-Lie algebra, $(B,\pi,\circ_B)$ is an action and $\partial:B\to A$ is a morphism of mock-Lie algebra, where
\begin{align*}
\partial (\pi(x)a) = x \circ_A \partial(a),\quad \pi(\partial(a))b = a \circ_B b,
\end{align*} for all $  x \in A, a, b \in B.$
Then, $\partial$ is a homomorphic  embedding tensor of the mock-Lie algebra $(A,\circ_A)$.
\end{ex}

In \cite{pei-rep}, the authors give the notion of Leibniz trialgebra. In the following,  we will introduce  anti-Leibniz trialgebra.

\begin{defi}\label{anti-Leib trialg} 
A (left) anti-Leibniz trialgebra is a triple $(A,\{\cdot,\cdot\},\circ )$ consisting of a mock-Lie algebra $(A,\circ )$ and an anti-Leibniz algebra $(A,\{\cdot,\cdot\})$ such that
    \begin{align}
& \{x,y\circ z\}+y\circ\{x,z\}+z\circ\{x,y\}=0 ,\label{eq:trileib1}\\
& \{x\circ y,z\}=\{\{x,y\},z\}, \label{eq:trileib2}
\end{align}
 for any $x,y, z \in A$, 
\end{defi}
\begin{rmk}\begin{enumerate}
    \item A right anti-Leibniz trialgebra is a triple $(A,\{\cdot,\cdot\},\circ )$ consisting of a mock-Lie algebra $(A,\circ )$ and a right anti-Leibniz algebra $(A,\{\cdot,\cdot\})$ such that, for any $x,y, z \in A$
    \begin{eqnarray}
& \{y\circ z,x\}+y\circ\{z,x\}+z\circ\{y,x\}=0 ,\label{eq:Rtrileib1}\\
& \{z, x\circ y\}=\{z,\{x,y\}\}.& \label{eq:Rtrileib2}
\end{eqnarray}

\item
    If we put $\{x,y\}^\prime=\{y,x\}$, then $(A,\{\c,\c\}^\prime,\circ )$ is a right anti-Leibniz trialgebra.
    \end{enumerate}
\end{rmk}
\begin{thm}\label{MLietoAntiLeibTri}
    Let $\mathcal{H}:B\to A$ be a homomorphic embedding tensor on the \textsf{MLie-Act} pair $(A,B)$. Then $(B,\{\c,\c\}_\mathcal{H},\circ _B)$ is a anti-Leibniz trialgebra, where
    $$\{a, b\}_\mathcal{H}=\pi(\mathcal{H}(a))b,\quad \forall\ a,b\in B.$$
\end{thm}
\begin{proof}
    We have already $(B,\circ _B)$ is a mock-Lie algebra. Moreover,  it follows from Theorem \ref{MLietoALeibniz} that $(B,\{\c,\c\}_\mathcal{H})$ is an anti-Leibniz algebra. 
    Now, for any $a,b,c\in B$, we have
    \begin{eqnarray*}
        &&\{a,b\circ_B c\}_\mathcal{H}+b\circ_B\{a,c\}_\mathcal{H}+c\circ_B\{a,b\}_\mathcal{H}\\
        &=&\pi(\mathcal{H}a)(b\circ_B c)+b\circ_B(\pi(\mathcal{H}a)c)
        + c\circ_B(\pi(\mathcal{H}a)b)\\
        & \overset{\eqref{actLie}}{=}&0.
    \end{eqnarray*}
    On the other hand, we have
    \begin{align*}
        \{a\circ_Bb,c\}_{\mathcal{H}}=&\pi(\mathcal{H}(a\circ_Bb))(c) \\
        =&\pi(\mathcal{H}(a)\circ\mathcal{H}(b))(c)\\
        =&\pi(\mathcal{H}(\pi(\mathcal{H}(a))b))(c)\\
        =&\{\{a,b\}_\mathcal{H},c\}_\mathcal{H}.
    \end{align*}
    Then,   Eqs \eqref{eq:trileib1} and \eqref{eq:trileib2} are satisfied. 
\end{proof}
\begin{ex}
    Let $A$ be a $3$-dimensional vector space with basis $\mathcal{B}=\{e_1,e_2,e_3\}$. Define the product $\circ$ by $e_1\circ e_1=e_2$ and the linear map $\mathcal{H}$ by
    $$\mathcal{H}:=\begin{pmatrix}
  1 &  0& 0 \\
  0 & 1 & 1\\
  0 & 0 & 0 \\
\end{pmatrix}.$$
Then $(A,\circ)$ is a mock-Lie algebra and $\mathcal{H}$ is a homomorphic embedding tensor on $A$. According to Theorem \ref{MLietoAntiLeibTri}, the  bracket $\{e_1,e_1\}=e_2$ together with the product $\circ$ endow $A$ with an anti-Leibniz trialgebra structure.
\end{ex}
On the direct sum vector space $A\oplus B$,  define the two bilinear maps $\circ _{A\oplus B}$ and $\{\cdot,\cdot\}_{A\oplus B}$  by
\begin{eqnarray}
(x+a)\circ_{A\oplus B}(y+b)&=&x
\circ_A y+a\circ_Bb,\\
    \{x+a,y+b\}_{A\oplus B}&=&x\circ_A y+\pi(x)b\quad\forall x,y\in A,a,b\in B.
\end{eqnarray}

\begin{pro}
  With the above notation, $(A\oplus B,\{\cdot,\cdot\}_{A\oplus B},\circ _{A\oplus B})$ is an anti-Leibniz trialgebra, called the hemisemidirect product anti-Leibniz trialgebra and denoted by $A\oplus_{ALeibT}B$.

\end{pro}
\begin{proof}
    It follows from Proposition \ref{hsdirect Lie} that $(A\oplus B,\{\c,\c\}_{A\oplus B})$ is a  anti-Leibniz algebra.
On the other hand, it is straightforward to observe that $(A\oplus B,\circ _{A\oplus B})$ is a mock-Lie algebra.
Let $x,y,z\in A$ and $a,b,c\in B$, according to Eq. \eqref{actLie}, we have
   {\small \begin{align*}
        &\{x+a,(y+b)\circ_{A\oplus B}(z+c)\}_{A\oplus B}+(y+b)\circ_{A\oplus B}\{x+a,z+c\}_{A\oplus B}+(z+c)\circ_{A\oplus B}\{x+a,y+b\}_{A\oplus B}\\=&x\circ_A (y\circ_A z)+\pi(x)(b\circ_Bc)+y\circ_A(x\circ_A z)+b\circ_B\pi(x)c+z\circ_A(x\circ_A y)+c\circ_B\pi(x)b\\=& 0.
    \end{align*}}
    Hence, the Eq. \eqref{eq:trileib1} is satisfied. The same for Eq. \eqref{eq:trileib2}. This completes the proof.
\end{proof}
\begin{thm}
Let $(A,B)$ be a \textsf{MLie-Act} pair. A linear map $\mathcal{H}:B\to A$ is a homomorphic embedding tensor on the mock-Lie algebra $(A,\circ_A)$ if and only if  the graph ${\sf Gr}(\mathcal{H}) = \{\mathcal{H}(a)+a\ |\ a \in B\}$ is a triLeibniz subantialgebra of the hemisemi-direct product anti-Leibniz trialgebra $A\oplus_{ALeibT}B$.
\end{thm}
\begin{proof}
Similar to Theorem \ref{graphMLie}.
\end{proof}
\begin{pro}
    Let $(A,B)$ be a \textsf{MLie-Act} pair. A linear map $\mathcal H: B \to A$ is a homomorphic embedding tensor on $B$ over $A$ if and only if the map
$$N_{\mathcal H}: (A \oplus B) \to (A \oplus B),\ (x+a) \mapsto \mathcal H(a),\quad \forall x\in A,a\in B,$$
is a Nijenhuis operator on the  hemisemi-direct product $A \oplus_{ALeibT} B$.
\end{pro}
\begin{defi}\label{antitriass}
An \emph{anti-associative trialgebra}  is a vector space $A$ that comes equipped with three binary operations, $\rprod$ (left), $\lprod$ (right), $\mprod$ (middle), where $(A,\lprod,\rprod)$ is an anti-associative dialgebra, $(A,\mprod)$ is an anti-associative algebra and for all $x, y, z \in A$:
\begin{align}
   &(x \rprod y) \rprod z + x \rprod (y \,\mprod\, z)=0, \label{axiom6} \\&
   (x \,\mprod\, y) \rprod z + x \,\mprod\, (y \rprod z)=0, \label{axiom7}\\&
   (x \rprod y) \,\mprod\, z + x \,\mprod\, (y \lprod z)=0, \label{axiom8} \\
   &(x \lprod y) \,\mprod\, z + x \lprod (y \,\mprod\, z)=0, \label{axiom9} \\&
   (x \,\mprod\, y) \lprod z + x \lprod (y \lprod z)=0. \label{axiom10}
   \end{align}
   \end{defi}
\begin{defi}
    Let $(A,\star_A)$ and $(B,\star_B)$ be two anti-associative algerbas. We will say that $A$ acts on $B$ if there is two linear maps $\varrho,\mu:A\rightarrow End(B)$, such that $(B,\varrho,\mu)$ is a representation of $(A,\star_A)$ and for any $x\in A$ and $a,b\in B$ we have
    \begin{align}\label{act}
        &\varrho(x)(a\star_B b)=-\varrho(x)a\star_B b,\qquad \mu(x)(a\star_B b)=-a\star_B \mu(x)b,\nonumber
        \\& \mu(x)a\star_B b=-a\star_B \varrho(x)b.
    \end{align}
    We denote an action of an anti-associative algebra by $(B,\varrho,\mu,\star_B)$.
\end{defi}
\begin{pro}
  The tuple $(B,\varrho,\mu,\star_B)$ is an action of an anti-associative algebra $(A,\star_A)$ if and only if $(A\oplus B,\star_{A\oplus B})$ is an anti-associative algebra, where
   $$(x+a)\star_{A\oplus B} (y+b)=x\star_A y+\varrho(x) b+\mu(y)a+a\star_B b,\quad\forall \ x,y\in A,\, a,b\in B.$$
\end{pro}
\begin{proof}
   Let $(A,\star_A)$ and $(B,\star_B)$ be two anti-associative algebras, then for any $x,y,z\in A$ and $a,b,c\in B$, we have
   \begin{align*}
       &((x+a)\star_{A\oplus B}(y+b))\star_{A\oplus B}(z+c)+(x+a)\star_{A\oplus B}((y+b)\star_{A\oplus B}(z+c))\\=&(x\star_A y)\star_A z+\varrho(x\star_A y)c+\mu(z)(\varrho(x)b+\mu(y)a+a\star_B b)+(\varrho(x)b+\mu(y)a+a\star_B b)\star_B c\\&+ x\star_A (y\star_A z)+\varrho(x)(\varrho(y)c+\mu(z)b+b\star_B c)+\mu(y\star_A z)a+a\star_B (\varrho(y)c+\mu(z)b+b\star_B c)\\=&\varrho(x\star_A y)c+\mu(z)(\varrho(x)b+\mu(y)a+a\star_B b)+(\varrho(x)b+\mu(y)a)\star_B c\\&+\varrho(x)(\varrho(y)c+\mu(z)b+b\star_B c)+\mu(y\star_A z)a+a\star_B (\varrho(y)c+\mu(z)b).
   \end{align*}
   Therefore, $(A\oplus V,\star_{A\oplus V})$ is an anti-associative algebra if and only if $(B,\varrho,\mu)$ is a representation of $(A,\star)$ and
   $$\varrho(x)(b\star_B c)=-\varrho(x)b\star_B c,\quad \mu(z)(a\star_B b)=-a\star_B \mu(z)b\quad \text{and}\quad \mu(y)a\star_B c=-a\star_B \varrho(y)c.$$
   Thus, $(A\oplus B,\star_{A\oplus B})$ is an anti-associative algebra if and only if $(B,\varrho,\mu,\star_B)$ is an action of $(A,\star_A)$.
\end{proof}

\begin{defi}
       A linear map $\mathcal H: B \to A$ is called a homomorphic embedding tensor on the anti-associative algebra $(A,\star_A)$ with respect to an action $(B,\varrho,\mu,\star_B)$ if it is both an embedding tensor and a morphism, that is
      \begin{align}\label{homembtensAA}
  & \mathcal{H}(a\star_Bb)=\mathcal H(a)\star_A \mathcal H(b),\ \forall\ a,b \in B.
      \end{align}
\end{defi}
 A homomorphic embedding tensor on the anti-associative algebra $(A,\star_A)$ with respect to the adjoint representation $(A,l,r)$ is called homomorphic averaging operator.
\begin{pro}
Let $(B,\varrho,\mu,\star_B)$ be an action of an anti-associative algebra $(A,\star_A)$. Then $(B,\pi=\varrho+\mu,\circ_B)$ is an action of the induced mock Lie algebra $(A,\circ_A)$, where $"\circ_A"$ and $"\circ_B"$ are the anti-commutator given in \eqref{anticom}.
    Moreover, if $\mathcal{H}:B\to A$ is a homomorphic embedding tensor on the anti-associative algebra $(A,\star_A)$ with respect to the action $(B,\varrho,\mu,\star_B)$. Then $\mathcal{H}$ is a homomorphic embedding tensor on the mock-Lie algebra $(A,\circ_A)$ with respect to the action $(B,\pi,\circ_B)$.
\end{pro}
\begin{proof}
According to Proposition \ref{embtensAAandMC} $(B,\pi)$ is a representation of $(A,\circ)$,  it remains to check \eqref{actLie}. For any $x\in A$ and $a,b\in B$, by \eqref{act} we have
\begin{align*}
   &\pi(x)(a\circ_B b)+a\circ_B \pi(x)b+b\circ_B \pi(x)a\\=&(\varrho+\mu)(x)(a\star_B b+b\star_B a)+a\star_B(\varrho+\mu)(x)b+(\varrho+\mu)(x)b\star_B a+b\star_B(\varrho+\mu)(x)a\\&+(\varrho+\mu)(x)a\star_B b\\=&\varrho(x)(a\star_B b)+\varrho(x)(b\star_B a)+\mu(x)(a\star_B b)+\mu(x)(b\star_B a)+a\star_B(\varrho(x)b)+a\star_B(\mu(x)b)\\&+(\varrho(x)b)\star_B a+(\mu(x)b)\star_B a+b\star_B(\varrho(x)a)+b\star_B(\mu(x)a)+(\varrho(x)a)\star_B b+(\mu(x)a)\star_B b\\=&0.
\end{align*}
On the other hand, if $\mathcal{H}$ is a homomorphic embedding tensor on the anti-associative algebra $(A,\star_A)$ with respect to $(B,\varrho,\mu,\star_B)$, then according to Proposition \ref{embtensAAandMC}  $\mathcal{H}$ is an embedding tensor on the mock Lie algebra $(A,\circ_A)$ with respect to $(B,\pi)$. Now, it remains to verify that $\mathcal{H}$ is also a morphism. For any $a,b\in B$, we have
\begin{align*}
    \mathcal{H}(a\circ_B b)&=\mathcal{H}(a\star_B b)+\mathcal{H}(b\star_B a)\\&=\mathcal{H}(a)\star_A \mathcal{H}(b)+\mathcal{H}(b)\star_A \mathcal{H}(a)\\&=\mathcal{H}(a)\circ_A\mathcal{H}(b).
\end{align*}
Hence, the result follows.
\end{proof}
\begin{cor}
    A homomorphic averaging operator $\mathcal{H}$ on a anti-associative algebra $(A,\star_A)$ is a homomorphic averaging operator on the induced mock-Lie algebra $(A,\circ_A)$.
\end{cor}
\begin{pro}\label{AntiAsstoAntiAssTri}
    Let $\mathcal H: B \to A$ be a homomorphic embedding tensor on the anti-associative algebra $(A,\star_A)$ with respect to an action $(B,\varrho,\mu,\star_B)$. Then, $B$ carries an anti-associative trialgebra structure given by \begin{align*}
        a\triangleright_{\mathcal{H}} b=\varrho(\mathcal{H}(a))b,\quad a\triangleleft_{\mathcal{H}} b=\mu(\mathcal{H}(b))a,\quad a\triangle_{\mathcal{H}} b=a \star_B b,\,  \forall \ a,b\in B. 
    \end{align*}   
\end{pro}
\begin{proof}
   We have that $(B,\triangle_{\mathcal{H}})$ is an anti-associative algebra and according to Proposition \ref{antiasstoantidiass}, $(B,\triangleright_{\mathcal{H}},\triangleleft_{\mathcal{H}})$ is an anti-associative dialgebra. It remains to check \eqref{axiom6}-\eqref{axiom10}, for any $a,b,c\in B$, we have
   \begin{eqnarray*}
(a\triangleleft_{\mathcal{H}}b)\triangleleft_{\mathcal{H}} c+a \triangleleft_{\mathcal{H}}(b\triangle_{\mathcal{H}} c)&=&\mu({\mathcal{H}}(c))(\mu({\mathcal{H}}(b))a)+\mu({\mathcal{H}}(b\star_B c))a\\
&=&\mu({\mathcal{H}}(c))(\mu({\mathcal{H}}(b))a)+\mu(\mathcal{H}(b) \star_A \mathcal{H}(c))a\\
&\overset{\eqref{repantiass}}{=}&0.
   \end{eqnarray*}
   On the other hand, we have
   \begin{eqnarray*}
       (a\triangle_{\mathcal{H}} b)\triangleleft_{\mathcal{H}} c+a\triangle_{\mathcal{H}} (b\triangleleft_{\mathcal{H}} c)&=&\mu(\mathcal{H} (c))(a\star_B b)+a\star_B (\mu(\mathcal{H} (c))b)\\&\overset{\eqref{act}}{=}&0.
   \end{eqnarray*}
   Similarly, one can check \eqref{axiom8}-\eqref{axiom10}. Hence, $(B,\triangleright_{\mathcal{H}},\triangleleft_{\mathcal{H}},\triangle_{\mathcal{H}})$ is an anti-associative trialgeba.
\end{proof}

\begin{thm}\label{triasstotriLeib}
    Let $(A,\lprod,\rprod,\mprod)$ be an anti-associative trialgebra. Define new binary operations
    \begin{align*}
        \{x, y\}&=x\lprod y+y\rprod x,\quad\quad x\circ y=x\mprod y+y\mprod x.
    \end{align*}
    Then, $(A,\{\c,\c\},\circ)$ is a anti-Leibniz trialgebra.
\end{thm}
\begin{proof}
    Since $(A,\lprod,\rprod,\mprod)$ is a anti-associative trialgebra, then $(A,\lprod,\rprod)$ is a anti-associative dialgebra. Thus, according to Theorem \ref{antiasstoantiLeib} $(A,\{\c,\c\})$ is an anti-Leibniz algebra.
Since $(A,\mprod)$ is a anti-associative algebra, then  $(A,\circ )$ is a mock-Lie algebra. It remains to 
check \eqref{eq:trileib1} and \eqref{eq:trileib2}.
 Using Definition \ref{antitriass}, then,  
 for any $x,y,z\in A$, we have
 \begin{eqnarray*}
     \{x,y\circ z\}+y\circ\{x,z\}+z\circ\{x,y\}&=&x\lprod (y\mprod z)+(y\mprod z)\rprod x+x\lprod (z\mprod y)\\&&+(z\mprod y)\rprod x+y\mprod(x\lprod z)+y\mprod(z\rprod x)\\&&+(x\lprod z)\mprod y+(z\rprod x)\mprod  y+z\lprod(x\mprod y)\\&&+z\mprod(y\rprod x)+(x\lprod y)\mprod z+(y\rprod x)\mprod z\\&=&0.
 \end{eqnarray*}
 On the other hand, we have
 \begin{eqnarray*}
     \{x\circ y,z\}-\{\{x,y\},z\}&=&(x\mprod y)\lprod z+(y\mprod x)\lprod z+z\rprod(x\mprod y)+z\rprod(y\mprod x)\\&&-(x\lprod y)\lprod z-(y\rprod x)\lprod z-z\rprod(x\lprod y)-z\rprod(y\rprod x)\\&=&-x\lprod(y\lprod z)-y\lprod(x\lprod z)-(z\rprod x)\rprod y-(z\rprod y)\rprod x\\&&-(x\lprod y)\lprod z-(y\rprod x)\lprod z-z\rprod(x\lprod y)-z\rprod(y\rprod x)\\&=&0.
     \end{eqnarray*}
     Thus, the result follows.
\end{proof}


\section{Further discussion}

In \cite{BenAliHajjejiChtiouiMabrouk} (see also \cite{Karima}), The authors present the concept of a mock-Lie bialgebra, which is equivalent to a Manin triple of mock-Lie algebras. In particular, they explore a specific case known as a coboundary mock-Lie bialgebra, which leads to the introduction of the mock-Lie Yang-Baxter equation, an analogue of the classical Yang-Baxter equation for Lie algebras. Notably, a skew-symmetric solution of the mock-Lie Yang-Baxter equation results in the formation of a mock-Lie bialgebra. In this paper, we study the dual representation. Motivated by this work, we will study anti-Leibniz bialgebras and show that matched
pairs of anti-Leibniz algebras, Manin triples of anti-Leibniz algebras and anti-Leibniz bialgebras are equivalent.
\section*{Acknowledgement}The authors would like to thank the referee for valuable comments and suggestions on this article.


\end{document}